\documentclass[12pt]{amsart}

\usepackage{mathrsfs,amssymb}

\newtheorem{theorem}{Theorem}[section]
\newtheorem{lemma}[theorem]{Lemma}

\theoremstyle{definition}
\newtheorem{definition}[theorem]{Definition}

\theoremstyle{remark}
\newtheorem{remark}[theorem]{Remark}

\numberwithin{equation}{section}

\let \la=\lambda
\let \e=\varepsilon

\let \o=\omega
\let \a=\alpha
\let \f=\varphi

\let \ga=\gamma

\begin{document}
\title[Coifman-Fefferman and Fefferman-Stein inequalities]
{A note on the Coifman-Fefferman and Fefferman-Stein inequalities}

\author{Andrei K. Lerner}
\address{Department of Mathematics,
Bar-Ilan University, 5290002 Ramat Gan, Israel}
\email{lernera@math.biu.ac.il}

\thanks{The author was supported by ISF grant No. 447/16 and ERC Starting Grant No. 713927.}

\begin{abstract}
A condition on a Banach function space $X$ is given under which the
Coifman-Fefferman and Fefferman-Stein inequalities on $X$ are equivalent.
\end{abstract}


\maketitle

\section{Introduction}
R. Coifman and C. Fefferman \cite{C,CF} proved that
\begin{equation}\label{CF}
\|Tf\|_{L^p(w)}\le C\|Mf\|_{L^p(w)},
\end{equation}
where $T$ is a singular integral operator and $M$ is the Hardy-Littlewood maximal operator. Approximately in the same period,
C. Fefferman and E. Stein \cite{FS} established that
\begin{equation}\label{FS}
\|Mf\|_{L^p(w)}\le C\|f^{\#}\|_{L^p(w)},
\end{equation}
where $f^{\#}$ is the sharp maximal operator.

Inequalities (\ref{CF}) and (\ref{FS}) play an important role in harmonic analysis. However,
the problem of characterizing the weights $w$ for which these inequalities hold is still open for both (\ref{CF}) and (\ref{FS}).
It is known that the so-called $C_p$ condition is necessary for (\ref{CF}) and (\ref{FS}), and $C_{p+\e},\e>0,$ is sufficient, see \cite{S,Y}.

The weak type versions of (\ref{CF}) and (\ref{FS}) (with the $L^p(w)$-norm replaced by the $L^{p,\infty}(w)$-norm on the left-hand side)
have been recently characterized in \cite{L6} by a uniform condition (denoted by $SC_p$).
In particular, this means that the weak type versions of (\ref{CF}) and (\ref{FS}) are equivalent.

Therefore, it is natural to conjecture that inequalities (\ref{CF}) and (\ref{FS}) are equivalent as well (of course when (\ref{CF}) is considered
on some natural subclass of non-degenerate singular integral operators). We investigate this question in a
general context of Banach function spaces (BFS) $X$ over ${\mathbb R}^n$ equipped with Lebesque measure.

Our main result says that the Coifman-Fefferman and Fefferman-Stein inequalities on $X$ are equivalent if
\begin{equation}\label{cond}
\|MMf\|_X\le C\|Mf\|_X.
\end{equation}
This condition is definitely superfluous if $X=L^p(w)$. However, we do not know if it can be removed (or at least weakened), in general.

In what follows, it will be convenient to introduce the space $MX$ equipped with norm
$$\|f\|_{MX}=\|Mf\|_{X}.$$
Then (\ref{cond}) means that the maximal operator $M$ is bounded on $MX$.

Denote by ${\mathfrak M}$ a family of BFS $X$ such that $MX$ is also BFS. It is easy to show (see Lemma \ref{pr} below) that
${\mathfrak M}$ consists of those $X$ for which $\f(x)=\frac{1}{1+|x|^n}\in X$.

Define the standard Riesz and the maximal Riesz transforms by
$$R_jf(x)=\lim_{\e\to 0}R_{j,\e}f(x)\quad\text{and}\quad R_j^{\star}f(x)=\sup_{\e>0}|R_{j,\e}f(x)|,$$
respectively, where
$$R_{j,\e}f(x)=c_n\int_{|y|>\e}f(x-y)\frac{y_j}{|y|^{n+1}}dy\quad(j=1,\dots,n).$$

Our main result is the following.

\begin{theorem}\label{mr}
Let $X$ be a BFS. Assume that $X\in {\mathfrak M}$ and that $M$ is bounded on $MX$. Then the following statements are equivalent.
\begin{enumerate}
\renewcommand{\labelenumi}{(\roman{enumi})}
\item
The Fefferman-Stein inequality
$$\|Mf\|_X\le C\|f^{\#}\|_X$$
holds.
\item
The Coifman-Fefferman inequality
$$\|R_j^{\star}f\|_X\le C\|Mf\|_{X}$$
holds for each of the maximal  Riesz transforms $R_j^{\star}, j=1,\dots, n$.
\end{enumerate}
\end{theorem}

The implication $(\rm{i})\Rightarrow(\rm{ii})$ here is easy; in fact, (i) along with the boundedness of $M$ on $MX$ implies that for every Calder\'on-Zygmund operator $T$,
$$\|T^{\star}f\|_X\le C\|Mf\|_{X}.$$

A non-trivial part of Theorem \ref{mr} is the implication $(\rm{ii})\Rightarrow(\rm{i})$. Here we essentially use a recent interesting result of D. Rutsky \cite{R} saying that for a BFS $X$,
the boundedness of every Riesz transform $R_{j}$ on $X$ is equivalent to the boundedness of $M$ on $X$ and $X'$, where $X'$ is the space associate to $X$.

A sketch of the proof of $(\rm{ii})\Rightarrow(\rm{i})$ is as follows.
First we show that the Fefferman-Stein inequality on $X$ is equivalent to the boundedness of $M$ on $(MX)'$, which is a slight refinement of an earlier characterization established in \cite{L2}.
Next, we obtain a pointwise estimate for the composition of $M$ with the maximal Calder\'on-Zygmund operator, which along with (ii) and the boundedness of $M$ on $MX$ implies that $R_j^{\star}$, and so $R_j,j=1,\dots n,$ are bounded on $MX$.
It remains to apply the above-mentioned result of D. Rutsky~\cite{R} to $MX$ instead of $X$.

Note that we will not need to use this result in full strength since the boundedness of $M$ on $MX$ is assumed. Therefore, in order to keep the paper essentially self-contained,
we give a simple proof of the fact that if $M$ and the Riesz transforms $R_j$ are bounded on $X$, then $M$ is bounded on $X'$.

Our method of the proof of $(\rm{ii})\Rightarrow(\rm{i})$ does not allow to replace the maximal Riesz
transforms $R_j^{\star}$ by the usual Riesz transforms $R_j$ in the statement of (ii). Also, it is an interesting question how to extend the implication
$(\rm{ii})\Rightarrow(\rm{i})$ to a more general class of singular integrals.

The paper is organized as follows. In Section 2, we recall some notions related to Banach function spaces and consider the space $MX$.
In Section 3, we obtain some pointwise estimates for Calder\'on-Zygmund operators. Section 4 contains a characterization of the Fefferman-Stein inequality.
Finally, in Section 5, we prove Theorem \ref{mr}.

\section{Banach function spaces and the space $MX$}
Denote by ${\mathcal M}^+$ the set of Lebesgue measurable non-negative functions on ${\mathbb R}^n$.

\begin{definition}\label{BFS}
By a Banach function space (BFS) $X$ over ${\mathbb R}^n$ equipped with Lebesque measure we mean a collection of functions $f$ such
that
$$\|f\|_{X}=\rho(|f|)<\infty,$$
where $\rho: {\mathcal M}^+\to [0,\infty]$ is a mapping satisfying
\begin{enumerate}
\renewcommand{\labelenumi}{(\roman{enumi})}
\item
$\rho(f)=0\Leftrightarrow f=0$ a.e.; $\rho(\a f)=\a \rho(f), \a\ge 0$;

\noindent
$\rho(f+g)\le \rho(f)+\rho(g)$;
\item $g\le f\,\,\text{a.e.}\,\,\Rightarrow \rho(g)\le \rho(f)$;
\item $f_n\uparrow f \,\,\text{a.e.}\,\,\Rightarrow \rho(f_n)\uparrow \rho(f)$;
\item if $E\subset {\mathbb R}^n$ is bounded, then $\rho(\chi_E)<\infty$;
\item if $E\subset {\mathbb R}^n$ is bounded, then $\int_Efdx\le c_E\rho(f)$.
\end{enumerate}
\end{definition}

A more common requirement is that $E$ is a set of finite measure in (iv) and (v) (see, e.g., \cite{BS}). However,
it is well known that all main elements of a general theory work with (iv) and (v) stated for bounded sets
(see, e.g., \cite{Lux}). In particular, if $X$ is a BFS, then the associate space $X'$ consisting of $f$ such that
$$\|f\|_{X'}=\sup_{g\in X:\|g\|_{X}\le 1}\int_{{\mathbb R}^n}|fg|\,dx<\infty$$
is also a BFS.

Recall that the Hardy-Littlewood maximal operator $M$ is defined by
$$Mf(x)=\sup_{Q\ni x}\frac{1}{|Q|}\int_Q|f(y)|dy,$$
where the supremum is taken over all cubes $Q$ containing $x$.

\begin{definition}\label{MX}
Assume that $X$ is a BFS. Denote by $MX$ the space of functions $f$ such that
$$\|f\|_{MX}=\|Mf\|_{X}<\infty.$$
\end{definition}

\begin{lemma}\label{pr} Assume that $X$ is a BFS. Then $MX$ is a BFS if and only if
\begin{equation}\label{ns}
\frac{1}{1+|x|^n}\in X.
\end{equation}
\end{lemma}

\begin{proof} We will use the standard fact that for every bounded set $E\subset {\mathbb R}^n$ of positive measure, there exist $C_1,C_2>0$
such that for all $x\in {\mathbb R}^n$,
\begin{equation}\label{two}
\frac{C_1}{1+|x|^n}\le M\chi_E(x)\le \frac{C_2}{1+|x|^n}.
\end{equation}

Assume that $MX$ is a BFS. Then by property (iv) of Definition \ref{BFS}, $\|M\chi_E\|_{X}<\infty$. Thus,
(\ref{ns}) follows from the left-hand side of (\ref{two}).

Assume now that (\ref{ns}) holds. Suppose that $\|f\|_{X}=\rho(|f|)$, where $\rho$ satisfies properties (i)-(v) of Definition \ref{BFS}.
Then $MX$ is a $BFS$ if
$$\rho'(f)=\rho(Mf)$$
satisfies these properties as well. Properties (i), (ii) and (v) are obvious. Next, (iv) follows from the right-hand side of (\ref{two}).

Finally, (iii) follows from the fact that if $f_n\uparrow f \,\,\text{a.e.}$, then $Mf_n \uparrow Mf$ everywhere.
Indeed, clearly $Mf_n$ is increasing and $Mf_n\le Mf$. Assume, for example, that $Mf(x)<\infty$. Let $\e>0$.
Take a cube $Q$ containing $x$ such that $Mf(x)<f_Q+\e$, where $f_Q=\frac{1}{|Q|}\int_Qf$. By Fatou's lemma, there exists $N$ such that for all $n\ge N$,
$$f_Q\le (f_n)_Q+\e\le Mf_n(x)+\e,$$
and hence $Mf(x)<Mf_n(x)+2\e$, which proves that $Mf_n(x)\to Mf(x)$ as $n\to \infty$. The case when $Mf(x)=\infty$ is similar.
\end{proof}

\section{Calder\'on-Zygmund operators}
Although Theorem \ref{mr} deals only with the Riesz transforms, some estimates we will use hold for general Calder\'on-Zygmund operators.

\begin{definition}\label{cz}
We say that $T$ is a Calder\'on-Zygmund operator with Dini-continuous kernel if $T$ is linear, $L^2$ bounded,
represented as
$$Tf(x)=\int_{{\mathbb R}^n}K(x,y)f(y)dy\quad\text{for all}\,\,x\not\in \text{supp}\,f$$
with kernel $K$ satisfying the size condition
$|K(x,y)|\le \frac{C_K}{|x-y|^n},x\not=y,$ and the smoothness condition
$$|K(x,y)-K(x',y)|+|K(y,x)-K(y,x')|\le \o\left(\frac{|x-x'|}{|x-y|}\right)\frac{1}{|x-y|^n}$$
for $|x-y|>2|x-x'|$,
where $\o$ is an increasing and subadditive on $[0,1]$ function such that $\int_0^1\o(t)\frac{dt}{t}<\infty.$
\end{definition}

We associate with $T$ the grand maximal truncated operator ${\mathcal M}_T$ and the usual maximal truncated operator $T^{\star}$ defined by
$${\mathcal M}_Tf(x)=\sup_{Q\ni x}\|T(f\chi_{{\mathbb R}^n\setminus 3Q})\|_{L^{\infty}(Q)}$$
and
$$T^{\star}f(x)=\sup_{\e>0}|T_{\e}f(x)|,$$
respectively, where $T_{\e}f(x)=\int_{|y-x|>\e}K(x,y)f(y)dy.$

It was shown in \cite{L4} that for every Calder\'on-Zygmund operator $T$ with Dini-continuous kernel,
\begin{equation}\label{mtr}
{\mathcal M}_Tf(x)\le C_{n,T}Mf(x)+T^{\star}f(x)
\end{equation}
for all $x\in {\mathbb R}^n$.
Exactly the same proof (replacing first $T$ by $T_{\e}$)  shows that $T$ on the left-hand side of (\ref{mtr}) can be replaced by $T^{\star}$, namely for all
$x\in {\mathbb R}^n$,
\begin{equation}\label{mtr1}
{\mathcal M}_{T^{\star}}f(x)\le C_{n,T}Mf(x)+T^{\star}f(x).
\end{equation}

\begin{lemma}\label{pt}
For every Calder\'on-Zygmund operator $T$ with Dini-continuous kernel and for all $x\in {\mathbb R}^n$,
$$M(T^{\star}f)(x)\le C_{n,T}MMf(x)+T^{\star}f(x).$$
\end{lemma}

\begin{proof} A variant of this estimate was obtained in \cite{L5}, and therefore we outline the proof briefly.
For every cube $Q$ containing the point $x$,
\begin{equation}\label{spl}
\frac{1}{|Q|}\int_Q
T^{\star}f\le \frac{1}{|Q|}
\int_QT^{\star}(f\chi_{3Q})
+\|T^{\star}(f\chi_{{\mathbb R}^n\setminus 3Q})\|_{L^{\infty}(Q)}.
\end{equation}

Since $T^{\star}$ is of weak type $(1,1)$ and $L^2$ bounded, interpolation along with Yano's extrapolation shows that the
first part on the right-hand side of (\ref{spl}) is controlled by $\|f\|_{L\log L(3Q)}$, which in turn is controlled by $MMf(x)$.
For the second part of (\ref{spl}) we use (\ref{mtr1}).
\end{proof}

\begin{remark}\label{blo}
Lemma \ref{pt} implies the well-known fact \cite{Le} that for every $f\in L^p\cap L^{\infty}$,
$$\|T^{\star}f\|_{BLO}\le C_{n,T}\|f\|_{L^{\infty}},$$
where $\|f\|_{BLO}=\|Mf-f\|_{L^{\infty}}$.
For usual (non-maximal) Calder\'on-Zygmund operators $T$ only a weaker property that $\|Tf\|_{BMO}\le C\|f\|_{L^{\infty}}$
holds, and by this reason $T^{\star}$ cannot be replaced by $T$ in the statement of Lemma \ref{pt}.
\end{remark}

Define the sharp maximal function $f^{\#}$ and the local sharp maximal function $M_{\la}^{\#}f$ respectively by
$$f^{\#}(x)=\sup_{Q\ni x}\frac{1}{|Q|}\int_Q|f-f_Q|dy$$
and
$$M_{\la}^{\#}f(x)=\sup_{Q\ni x}\inf_c\big((f-c)\chi_Q\big)^*\big(\la|Q|\big)\quad(0<\la<1),$$
where $f^*$ denotes the non-increasing rearrangement of $f$.

It was shown in \cite{JT} that for all $x\in {\mathbb R}^n$,
$$
C_1MM^{\#}_{\la}f(x)\le f^{\#}(x)\le C_2MM^{\#}_{\la}f(x) \quad(0<\la\le 1/2)
$$
and for every Calder\'on-Zygmund operator $T$ with Dini-continuous kernel,
$$
M_{\la}^{\#}(T^{\star}f)(x)\le CMf(x)
$$
(actually the latter estimate was proved in \cite{JT} for $T$ instead of $T^{\star}$ but the proof for $T^{\star}$ is essentially the same).
A combination of these two estimates yields
\begin{equation}\label{pot}
(T^{\star}f)^{\#}(x)\le CMMf(x).
\end{equation}

\section{A characterization of the Fefferman-Stein inequality}
Consider the Fefferman-Stein inequality
\begin{equation}\label{fs}
\|Mf\|_{X}\le C\|f^{\#}\|_X.
\end{equation}

It was shown in \cite{L2} that (\ref{fs}) is actually equivalent to the same estimate but with $Mf$ replaced by $f$ on the left-hand side, namely, (\ref{fs}) holds if and only if
\begin{equation}\label{fs1}
\|f\|_{X}\le C\|f^{\#}\|_X.
\end{equation}

Also, it was shown in \cite{L2} that (\ref{fs}) holds if and only if
$$
\int_{{\mathbb R}^n}(Mf)g\,dx\le C\|f\|_{X'}\|Mg\|_{X}.
$$
This estimate can be rewritten in the form
\begin{equation}\label{fs2}
\|Mf\|_{(MX)'}\le C\|f\|_{X'}.
\end{equation}

Here we show that essentially the same proof as in \cite{L2} yields a bit more precise version of (\ref{fs2}) with
$\|f\|_{X'}$ on the right-hand side replaced by a smaller expression $\|f\|_{(MX)'}$.

\begin{theorem}\label{fsc} Let $X$ be a Banach function space such that $X\in {\mathfrak M}$. The Fefferman-Stein inequality (\ref{fs}) holds
if and only if the maximal operator $M$ is bounded on $(MX)'$.
\end{theorem}

\begin{remark}\label{rem}  Theorem \ref{fsc} says that exactly as (\ref{fs1}) can be self-improved to (\ref{fs}), the estimate
(\ref{fs2}) can be self-improved to
$$\|Mf\|_{(MX)'}\le C\|f\|_{(MX)'}.$$
\end{remark}

\begin{proof}[Proof of Theorem \ref{fsc}]
Since $\|f\|_{(MX)'}\le \|f\|_{X'}$, the boundedness of $M$ on $(MX)'$ implies (\ref{fs2}), which in turns implies (\ref{fs})
by a characterization obtained in \cite{L2}.

To show the converse direction, we use a result by A. de la Torre~\cite{T} saying that for every locally integrable $f$, there is a linear operator ${\mathcal M}_f$ such that $Mf$ is pointwise equivalent to ${\mathcal M}_ff$ and for every locally integrable $g$,
$$({\mathcal M}_f^*g)^{\#}(x)\le CMg(x),$$
where ${\mathcal M}_f^*$ is the adjoint of ${\mathcal M}_f$. From this, for every $f,g\ge 0$ we obtain
\begin{eqnarray*}
&&\int_{{\mathbb R}^n}(Mf)g\le C\int_{{\mathbb R}^n}({\mathcal M}_ff)g=C\int_{{\mathbb R}^n}f{\mathcal M}_f^{\star}g\\
&&\le C\|f\|_{(MX)'}\|M{\mathcal M}_f^{\star}g\|_{X}
\le C\|f\|_{(MX)'}\|({\mathcal M}_f^*g)^{\#}\|_{X}\\
&&\le C\|f\|_{(MX)'}\|Mg\|_{X}=C\|f\|_{(MX)'}\|g\|_{MX},
\end{eqnarray*}
which, by duality, implies that $M$ is bounded on $(MX)'$.
\end{proof}

\section{Proof of Theorem \ref{mr}}
As we have mentioned in the Introduction, Theorem \ref{mr} can be deduced by combining the ingredients from Sections 3 and 4 with the result
of D. Rutsky \cite{R}. We give a simplified proof of a weaker version of this result, which is enough for our purposes.
We start with the following lemma.

\begin{lemma}\label{ra}
Assume that  for all $x\in {\mathbb R}^n$,
\begin{equation}\label{condr}
|R_j\f(x)|\le K\f(x)\quad(j=1,\dots,n).
\end{equation}
Then there exists $0<q<1$ depending only on $n$ such that
$$M\f(x)\le c_nKM(\f^q)(x)^{1/q}$$
for all $x\in {\mathbb R}^n$.
\end{lemma}

\begin{remark}\label{flr} A conclusion of this lemma can be refined. In fact, it was shown by I. Vasilyev \cite{V} that
(\ref{condr}) implies $\log \f\in BMO$. Moreover, D. Rutsky \cite{R} observed that (\ref{condr}) implies $\f\in A_{\infty}$.

We also note that M. Cotlar and C. Sadosky \cite{CS} showed that (\ref{condr}) for the Hilbert transform on the unit circle implies $\f\in A_2$. 
We give a simple proof of this fact for the Hilbert transform $H$ on the real line. 

Assume that $|H\f|\le K\f$. Applying the ``magic identity" $(Hf)^2=f^2+2H(fHf)$, we obtain
\begin{equation}\label{mag}
\int_{{\mathbb R}}(Hf)^2\f=\int_{{\mathbb R}}f^2\f+2\int_{{\mathbb R}}H(fHf)\f.
\end{equation}
Further,
\begin{eqnarray*}
\Big|\int_{{\mathbb R}}H(fHf)\f\Big|=\Big|\int_{{\mathbb R}}f(Hf)H(\f)\Big|&\le& K\int_{{\mathbb R}}|f||Hf|\f\\
&\le&K\|f\|_{L^2(\f)}\|Hf\|_{L^2(\f)}.
\end{eqnarray*}
Combining this with (\ref{mag}) yields
$$\|Hf\|_{L^2(\f)}\le C_K\|f\|_{L^2(\f)},$$
which implies $\f\in A_2$.
\end{remark}

\begin{proof}[Proof of Lemma \ref{ra}] Let $P_t$ be the Poisson kernel. Denote
$$u_j(x,t)=(R_j(\f)*P_t)(x), j=1,\dots,n,\quad u_{n+1}(x,t)=(\f*P_t)(x)$$
and $F=(u_1,\dots,u_{n+1})$. It is well known \cite[p. 143]{G} that $|F|^q$ is subharmonic when $q\ge \frac{n-1}{n}$.
In particular, this implies (see \cite[p. 145]{G}) that for all $x\in {\mathbb R}^n$ and $t,\e>0$,
$$|F(x,t+\e)|^q\le (|F(\cdot,\e)|^q*P_t)(x).$$
Passing to the limit when $\e\to 0$ and applying (\ref{condr}), we obtain
$$|\f*P_t(x)|^q\le |F(x,t)|^q\le K^q(n+1)^{q/2}(\f^q)*P_t(x).$$
Taking here the supremum over $t>0$ completes the proof.
\end{proof}

The following lemma is the above mentioned weaker version of the result in \cite{R}.

\begin{lemma}\label{sr}
Let $X$ be a BFS. Assume that the maximal operator $M$ and that every Riesz transform $R_j, j=1,\dots, n$ are bounded on $X$. Then $M$ is bounded on $X'$.
\end{lemma}

\begin{proof}
By duality, the Riesz transforms $R_j$ are bounded on $X'$.
Denote
$$Rf(x)=\sum_{j=1}^n|R_jf(x)|,$$
and consider the following Rubio de Francia type operator
$$S_{R}f(x)=\sum_{k=0}^{\infty}\frac{1}{(2\ga)^k}\f_k(x),$$
where $\ga=\|R\|_{X'\to X'}, \f_0=|f|$ and $\f_k=R(\f_{k-1}), k\in {\mathbb N}$. Then
\begin{equation}\label{rdf}
R(S_Rf)(x)\le 2\ga S_Rf(x)
\end{equation}
and also
\begin{equation}\label{rdf1}
|f|\le S_{R}f \quad\text{and}\quad \|S_{R}f\|_{X'}\le 2\|f\|_{X'}.
\end{equation}

Combining (\ref{rdf}) with Lemma \ref{ra} yields
$$M(S_Rf)(x)\le c_n\ga M((S_Rf)^q)^{1/q}(x),$$
where $q=q(n)<1$. Therefore, by the Fefferman-Stein inequality \cite{FS1} along with (\ref{rdf1}),
\begin{eqnarray*}
\int_{{\mathbb R}^n}(Mf)g\,dx&\le& \int_{{\mathbb R}^n}(M(S_Rf))g\,dx\le c_n\ga \int_{{\mathbb R}^n}M((S_Rf)^q)^{1/q}g\,dx\\
&\le& c_n'\ga\int_{{\mathbb R}^n}(S_Rf)(Mg)dx\le c_n'\ga\|S_Rf\|_{X'}\|Mg\|_{X}\\
&\le&C\|f\|_{X'}\|g\|_{X},
\end{eqnarray*}
which implies that $M$ is bounded on $X'$.
\end{proof}

\begin{proof}[Proof of Theorem \ref{mr}]
The implication $(\rm{i})\Rightarrow(\rm{ii})$ follows from (\ref{pot}) and from the assumption that $M$ is bounded on $MX$.

Turn to the implication $(\rm{ii})\Rightarrow(\rm{i})$. Applying Lemma \ref{pt} and using the boundedness of $M$ on $MX$, we obtain
$$\|M(R_j^{\star}f)\|_{X}\le C\|MMf\|_{X}+\|R_j^{\star}f\|_{X}\le C\|Mf\|_{X}\quad(j=1,\dots,n).$$
This means that every $R_{j}^{\star}$ is bounded on $MX$, and therefore $R_{j}$ is bounded on $MX$ as well.
By Lemma \ref{sr} we obtain that $M$ is bounded on $(MX)'$, which, by Theorem \ref{fsc}, is equivalent to that the Fefferman-Stein inequality holds on $X$.
\end{proof}

\end{document}